\documentclass{amsart}

\usepackage{amssymb}

\usepackage[matrix,arrow,curve]{xy}

\theoremstyle{plain}

\newtheorem{theorem}{Theorem}[section]

\newtheorem{corollary}[theorem]{Corollary}

\newtheorem{proposition}[theorem]{Proposition}

\theoremstyle{definition} 

\newtheorem{definition}[theorem]{Definition}

\theoremstyle{remark}

\newtheorem{equa}[theorem]{Equality}

\newtheorem{exercise}[theorem]{Exercise}

\renewcommand{\phi}{\varphi}

\newcommand{\initial}\lessdot

\def\?{?\vadjust

{\vbox to 0pt{\vskip-7pt\hbox to 1.1\hsize{\hfill\huge ?!}}}}

\newcommand{\be}{\begin{enumerate}}
\newcommand{\ee}{\end{enumerate}}

\usepackage{enumerate}
\usepackage{verbatim}
\renewcommand{\epsilon}{\varepsilon}

\usepackage{tikz}

 \def\nfork{\setbox0\hbox{$\bigcup$}%
 \setbox1=\hbox to \wd0{\hfil\vrule width 0.7pt depth 2pt height 7.5pt\hfil}%
 \wd1=0cm\relax\box1\box0}

\begin{document}

\title{Duality and Hereditary K\"onig-Egerv\'ary Set-systems}

\author{Adi Jarden}
\email[Adi Jarden]{jardena@ariel.ac.il}
\address{Department of Mathematics.\\ Ariel University \\ Ariel, Israel}

\maketitle

\today

\begin{abstract}
A K\"onig-Egerv\'ary graph is a graph $G$ satisfying $\alpha(G)+\mu(G)=|V(G)|$, where $\alpha(G)$ is the cardinality of a maximum independent set and $\mu(G)$ is the matching number of $G$. Such graphs are those that admit a matching between $V(G)-\bigcup \Gamma$ and $\bigcap \Gamma$ where $\Gamma$ is a set-system comprised of maximum independent sets satisfying $|\bigcup \Gamma'|+|\bigcap \Gamma'|=2\alpha(G)$ for every set-system $\Gamma' \subseteq \Gamma$; in order to improve this characterization of a K\"onig-Egerv\'ary graph, we characterize \emph{hereditary K\"onig-Egerv\'ary set-systems} (HKE set-systems, here after).

An \emph{HKE} set-system is a set-system, $F$, such that for some positive integer, $\alpha$, the equality
$|\bigcup \Gamma|+|\bigcap \Gamma|=2\alpha$ holds for every non-empty subset, $\Gamma$, of $F$.

We prove the following theorem:
Let $F$ be a set-system. $F$ is an HKE set-system if and only if the equality $|\bigcap \Gamma_1-\bigcup \Gamma_2|=|\bigcap \Gamma_2-\bigcup \Gamma_1|$ holds for every two non-empty disjoint subsets, $\Gamma_1,\Gamma_2$ of $F$.

This theorem is applied in \cite{hke},\cite{broken}.
\end{abstract}

\section{Introduction}
In this section we give the basic definitions and motivate the study of HKE set-systems.

For a uniform set-system, $F$, we denote by $\alpha(F)$ the cardinality of a set in $F$. We write $\alpha$, when $F$ is clear from the context.

The following definition contradicts the definition of a K\"onig-Egerv\'ary set-system in \cite{dam}.
\begin{definition}
Let $F$ be a uniform set-system. $F$ is said to be a \emph{K\"onig-Egerv\'ary set-system} (KE set-system in short), if the following equality holds: $$|\bigcup F|+|\bigcap F|=2\alpha(F).$$
\end{definition}

\begin{definition}\label{definition of a KE set-system}
An \emph{HKE} set-system is a set-system, $F$, such that for some positive integer, $\alpha$, the equality
$$|\bigcup \Gamma|+|\bigcap \Gamma|=2\alpha$$ holds for every non-empty subset, $\Gamma$, of $F$.
\end{definition}

\begin{proposition}\label{1}
Every HKE set-system is a uniform set-system. So a set-system $F$ is HKE if and only if each subset $\Gamma$ of $F$ is KE. 
\end{proposition}

\begin{proof}
Let $F$ be an HKE set-system and let $A \in F$. By Definition \ref{definition of a KE set-system}, where we substitute $\Gamma=\{A\}$, we have $|A|=\alpha$. So $F$ is a uniform set-system and $\alpha=\alpha(F)$.
\end{proof}

\begin{proposition}\label{2}
Let $F$ be a uniform set-system. If $|F| \leq 2$ then it is an HKE set-system.
\end{proposition}

\begin{proof}
It is clear when $|F|=1$. So assume $|F|=2$,
 $F=\{A,B\}$. Take a non-empty sub-set-system $\Gamma$ of $F$. Without loss of generality, $\Gamma=F$. So $$|\bigcup \Gamma|+|\bigcap \Gamma|=|A \cup B|+|A \cap B|=|A|+|B|=2\alpha(F).$$
\end{proof}

Theorem \ref{the main theorem of dam} and Propositions \ref{the first on 2016},\ref{proposition 1.7} exemplifies the usefullness of HKE set-systems in the study of K\"onig-Egerv\'ary graphs.

The following theorem is a restatement of \cite[Theorem 2.6]{dam} in our notation.
\begin{theorem}\label{the main theorem of dam}
$G$ is a K\"onig-Egerv\'ary graph if and only if there is a matching between $V(G)-\bigcup \Gamma$ and $\bigcap \Gamma$, where $\Gamma$ is an HKE set-system comprised of maximum independent sets.
\end{theorem}

\begin{proposition}\label{the first on 2016}
Let $G$ be a KE graph. Then $\Omega(G)$ is an HKE set-system.
\end{proposition}

\begin{proof}
By \cite[Theorem 3.6]{jlm} and \cite[Corollary 2.8]{jlm}.
\end{proof}

\begin{proposition}\label{proposition 1.7}
Every KE set-system that is comprised of maximum independent sets of some graph is an HKE set-system.  
\end{proposition}

\begin{proof}
By \cite[Corollaries 2.7 and 2.9]{jlm}.
\end{proof}

\section{HKE set-systems and duality}
In this section, we characterize the HKE set-systems; consequently, we get a new characterization of a K\"onig-Egerv\'ary graph. Proposition \ref{the theorem before 23.3.16} is a weak version of Theorem \ref{equivalent formulations of a KE set-system}, where we add the assumption, that the set-system is uniform.
 
In order to state Proposition \ref{the theorem before 23.3.16}, Theorem \ref{equivalent formulations of a KE set-system} and Corollary \ref{corollary 2.6}, we present the following equality:
\begin{equa}\label{equality 1}
$$|\bigcap \Gamma_1-\bigcup \Gamma_2|=|\bigcap \Gamma_2-\bigcup \Gamma_1|.$$
\end{equa}

\begin{proposition}\label{the theorem before 23.3.16} 
Let $F$ be a uniform set-system. 

 The following are equivalent:
\begin{enumerate}
\item $F$ is an HKE set-system. \item Equality \ref{equality 1} holds for every two non-empty disjoint sub-set-systems, $\Gamma_1,\Gamma_2$ of $F$, \item Equality \ref{equality 1} holds for every two non-empty disjoint sub-set-systems, $\Gamma_1,\Gamma_2$ of $F$ with $\Gamma_1 \cup \Gamma_2=F$.
\end{enumerate}
\end{proposition}

The argument of Proposition \ref{the theorem before 23.3.16} is based on the following exercise:
\begin{exercise}\label{exercise 1}
Assume that $\{A,B,C,D\}$ is an HKE set-system (so in particular $\{A,B,C\}$ is an HKE set-system).
Prove:
\begin{enumerate}
\item $|A-B-C|=|B \cap C-A|$. A clue: $A-B-C=(A \cup B \cup C)-(B \cup C)$ and $B \cap C-A=(B \cap C)-(A \cap B \cap C)$. \item $|A \cap B-C-D|=|C \cap D-A-B|$. A clue: $A \cap B -C-D=(A-C-D)-(A-B-C-D)$. Apply Clause (1).
\end{enumerate}
\end{exercise}

We now prove Proposition \ref{the theorem before 23.3.16}.
\begin{proof}
$(1) \Rightarrow (2):$
We prove it by induction on $r=:|\Gamma_1|$. 

\emph{Case a:} $r=1$, so $\Gamma_1=\{A^*\}$ for some set $A^*$. In this case, we apply the idea of Exercise \ref{exercise 1}(1).

We should prove that $$|A^*-\bigcup \Gamma_2|=|\bigcap \Gamma_2-A^*|,$$ 

namely, $$|\bigcup \Gamma_2 \cup A^*|-|\bigcup \Gamma_2|=|\bigcap \Gamma_2|-|\bigcap \Gamma_2 \cap A^*|,$$
or equivalently,
$$|\bigcup \Gamma_2 \cup A^*|+|\bigcap \Gamma_2 \cap A^*|=|\bigcap \Gamma_2|+|\bigcup \Gamma_2|.$$
 But by Clause $(1)$, each side of this equality equals $2\alpha$.

\emph{Case a:} $r>1$. In this case, we apply the idea of Exercise \ref{exercise 1}(2). We fix $A^* \in \Gamma_1$. First we write three trivial equalities, for convenience: $$\bigcap (\Gamma_1-\{A^*\})=\{x:x \in A \text{ for every } A \in \Gamma_1 \text{ with } A \neq A^*\},$$ $$\bigcup (\Gamma_1-\{A^*\})=\{x:x \in A \text{ for some } A \in \Gamma_1 \text{ with } A \neq A^*\}$$ and $$\bigcap (\Gamma_1 \cup \{A^*\})=A^* \cap \bigcap \Gamma_1.$$ 

We now begin the computation.

$$|\bigcap \Gamma_1-\bigcup \Gamma_2|=|\bigcap (\Gamma_1-\{A^*\})-\bigcup \Gamma_2|-|\bigcap (\Gamma_1-\{A^*\})-\bigcup (\Gamma_2 \cup \{A^*\})|.$$

The right side of this equality is a subtraction of two summands. Since $|\Gamma_1-\{A^*\}|<|\Gamma_1|$, we may apply the induction hypothesis on each summand: 
	
	$$|\bigcap (\Gamma_1-\{A^*\})-\bigcup \Gamma_2|=|\bigcap \Gamma_2-\bigcup (\Gamma_1-\{A^*\})|$$ and
	
	$$|\bigcap (\Gamma_1-\{A^*\})-\bigcup (\Gamma_2 \cup \{A^*\})|=|\bigcap (\Gamma_2 \cup \{A^*\})-\bigcup (\Gamma_1-\{A^*\})|.$$
	
By the three last equalities we get: $$|\bigcap \Gamma_1-\bigcup \Gamma_2|=|\bigcap \Gamma_2-\bigcup (\Gamma_1-\{A^*\})|-|\bigcap (\Gamma_2 \cup \{A^*\})-\bigcup (\Gamma_1-\{A^*\})|.$$

So $$|\bigcap \Gamma_1-\bigcup \Gamma_2|=|\bigcap \Gamma_2-\bigcup \Gamma_1|.$$

Equality \ref{equality 1} is proved, so Clause $(2)$ is proved.

$(2) \Rightarrow (1):$
Let $\Gamma$ be a non-empty subset of $F$.
Fix $D \in \Gamma$. Since $F$ is a uniform set-system, $|D|=\alpha$ (this is the unique place where we use the assumption that $F$ is a uniform set-system, but we eliminate this assumption later). Therefore it is enough to prove that
$$|\bigcup \Gamma|+|\bigcap \Gamma|=2|D|,$$
or equivalently, $$|\bigcup \Gamma-D|=|D-\bigcap \Gamma|.$$
Let $H$ be the set of ordered pairs $\langle \Gamma_1,\Gamma_2 \rangle$ of non-empty disjoint subsets of $\Gamma$ such that $\Gamma_1 \cup \Gamma_2=\Gamma$ and $D \in \Gamma_2\}$.

By Clause (2), $$\sum_{\langle \Gamma_1,\Gamma_2 \rangle \in H} |\bigcap \Gamma_1-\bigcup \Gamma_2|=\sum_{\langle \Gamma_1,\Gamma_2 \rangle \in H} |\bigcap \Gamma_2-\bigcup \Gamma_1|.$$

So it is enough to prove the following two equalities:
$$|\bigcup \Gamma-D|=\sum_{\langle \Gamma_1,\Gamma_2 \rangle \in H} |\bigcap \Gamma_1-\bigcup \Gamma_2|$$
and
$$|D-\bigcap \Gamma|=\sum_{\langle \Gamma_1,\Gamma_2 \rangle \in H} |\bigcap \Gamma_2-\bigcup \Gamma_1|.$$

Since their proofs are dual, we prove the first equality only.

$$\bigcup \Gamma-D=\bigcup_{\langle \Gamma_1,\Gamma_2 \rangle \in H} (\bigcap \Gamma_1-\bigcup \Gamma_2),$$

(on the one hand, if $x \in \bigcup \Gamma-D$ then for $\Gamma_1=\{A \in \Gamma:x \in A\}$ and $\Gamma_2=\{A \in \Gamma:x \notin A\}$ we have $x \in \bigcap \Gamma_1-\bigcup \Gamma_2$ and $\langle \Gamma_1,\Gamma_2 \rangle \in H$. On the other hand, assume that $x \in \bigcap \Gamma_1-\bigcup \Gamma_2$ for some $\langle \Gamma_1,\Gamma_2 \rangle \in H$. Then $x \in \bigcup \Gamma$ (because $x \in \bigcap \Gamma_1$ and $\emptyset \neq \Gamma_1 \subseteq \Gamma$) and $x \notin D$ (because $x \notin \bigcup \Gamma_2$ and $D \subseteq \bigcup \Gamma_2$). So $x 
\in \bigcup \Gamma-D$). Therefore

$$|\bigcup \Gamma-D|=\sum_{\langle \Gamma_1,\Gamma_2 \rangle \in H} |\bigcap \Gamma_1-\bigcup \Gamma_2|,$$
because this is a sum of cardinalities of disjoint sets (if $\langle \Gamma_1,\Gamma_2 \rangle$ and $\langle \Gamma_3,\Gamma_4 \rangle$ are two different pairs in $H$ then there is no element $x \in (\bigcap \Gamma_1-\bigcup \Gamma_2) \cap (\bigcap \Gamma_3-\bigcup \Gamma_4)$. Otherwise, take $A \in \Gamma_1-\Gamma_3$ (or vice versa). So $A \in \Gamma_4$. Hence, $x \in \bigcap \Gamma_1 \subseteq A$ and $x \notin \bigcup \Gamma_4 \supseteq A$, a contradiction).

The implication $(2) \Rightarrow (1)$ is proved.

Since Clause $(3)$ is a private case of Clause $(2)$, it remains to prove $(3) \Rightarrow (2)$. Let $\Gamma_1,\Gamma_2$ be two non-empty disjoint subsets of $F$. We should prove Equality \ref{equality 1} for these $\Gamma_1$ and $\Gamma_2$, without assuming $\Gamma_1 \cup \Gamma_2=F$. Let $H$ be the set of disjoint pairs $\langle \Gamma_1^+,\Gamma_2^+ \rangle$ of $F$ such that $\Gamma_1 \subseteq \Gamma_1^+$, $\Gamma_2 \subseteq \Gamma_2^+$ and $\Gamma_1^+ \cup \Gamma_2^+=F$.

By Clause $(3)$, 
$$\sum_{\langle \Gamma_1^+,\Gamma_2^+ \rangle \in H} |\bigcap \Gamma_1^+-\bigcup \Gamma_2^+|=\sum_{\langle \Gamma_1^+,\Gamma_2^+ \rangle \in H} |\bigcap \Gamma_2^+-\bigcup \Gamma_1^+|.$$

So it remains to prove the following two equalities:
$$|\bigcap \Gamma_1-\bigcup \Gamma_2|=\sum_{\langle \Gamma_1^+,\Gamma_2^+ \rangle \in H} |\bigcap \Gamma_1^+-\bigcup \Gamma_2^+|$$ 
and
$$|\bigcap \Gamma_2-\bigcup \Gamma_1|=\sum_{\langle \Gamma_1^+,\Gamma_2^+ \rangle \in H} |\bigcap \Gamma_2^+-\bigcup \Gamma_1^+|,$$

Since their proofs are dual, we prove the first equality only.

$$\bigcap \Gamma_1-\bigcup \Gamma_2=\bigcup_{\langle \Gamma_1^+,\Gamma_2^+ \rangle \in H} (\bigcap \Gamma_1^+-\bigcup \Gamma_2^+)$$

(On the one hand, if $x \in \bigcap \Gamma_1-\bigcup \Gamma_2$ then for $\Gamma_1=\{A \in \Gamma:x \in A\}$ and $\Gamma_2=\{A \in \Gamma:x \notin A\}$, we have $x \in \bigcap \Gamma_1^+-\bigcup \Gamma_2^+$ and the pair $\langle \Gamma_1^+,\Gamma_2^+ \rangle$ belongs to $H$. On the other hand, if $x \in \bigcap \Gamma_1^+-\bigcup \Gamma_2^+$ for some $\langle \Gamma_1^+,\Gamma_2^+ \rangle \in H$ then $x \in \bigcap \Gamma_1^+ \subseteq \bigcap \Gamma_1$ and $x \notin \bigcup \Gamma_2^+ \supseteq \bigcup \Gamma_2$. Hence, $x \in \bigcap \Gamma_1-\bigcup \Gamma_2$). Therefore

$$|\bigcap \Gamma_1-\bigcup \Gamma_2|=\sum_{\langle \Gamma_1^+,\Gamma_2^+ \rangle \in H} |\bigcap \Gamma_1^+-\bigcup \Gamma_2^+|,$$
because it is a sum of disjoint sets.
\end{proof}

The following proposition eliminates the assumption that $F$ is a uniform set-system.
\begin{proposition}\label{elimination of relevant}
Clause (3) of Proposition \ref{the theorem before 23.3.16} implies that $F$ is a uniform set-system. 
\end{proposition}

\begin{proof}
Define $$\alpha=\frac{|\bigcup F|+|\bigcap F|}{2}.$$ Let $D \in F$. We prove that $|D|=\alpha$. Let $P$ denote the family of partitions $\{\Gamma_1,\Gamma_2\}$ of $F$ into two non-empty subsets.  

Every element in $\bigcup F$ is in $\bigcap \Gamma_1-\bigcup \Gamma_2$ for some partition $\{\Gamma_1,\Gamma_2\} \in P$ or in $\bigcap F$.

Let $$P_1=\{\{\Gamma_1,\Gamma_2\} \in P:D \in \Gamma_1\}$$ and $$P_2=\{\{\Gamma_1,\Gamma_2\} \in P:D \notin \Gamma_1\}.$$ 

Define $$x=\sum_{\{\Gamma_1,\Gamma_2\} \in P_1}|\bigcap \Gamma_1-\bigcup \Gamma_2|$$
and
$$y=\sum_{\{\Gamma_1,\Gamma_2\} \in P_2}|\bigcap \Gamma_1-\bigcup \Gamma_2|.$$

By Clause $(3)$ of Proposition \ref{the theorem before 23.3.16}, we have $x=y$.

It is easy to check the following three equalities:
\begin{enumerate}
\item $|\bigcup F|=x+y+|\bigcap F|=2x+|\bigcap F|$, \item
$|D|=x+|\bigcap F|$ and \item
$|\bigcup F|+|\bigcap F|=2\alpha$ (by the definition of $\alpha$).
\end{enumerate}

By Equalities (1)-(3), $|D|=\alpha$. Since $D$ is an arbitrary set in $F$, $F$ is a uniform set-system. 
\end{proof}

\begin{theorem}\label{equivalent formulations of a KE set-system}\label{equivalent definitions of a KE set-system}\label{equivalent definitions of an HKE set-system}
Let $F$ be a set-system. 

The following are equivalent:
\begin{enumerate}
\item $F$ is an HKE set-system. \item Equality \ref{equality 1} holds for every two non-empty disjoint sub-set-systems, $\Gamma_1,\Gamma_2$ of $F$, \item Equality \ref{equality 1} holds for every two non-empty disjoint sub-set-systems, $\Gamma_1,\Gamma_2$ of $F$ with $\Gamma_1 \cup \Gamma_2=F$.
\end{enumerate}
\end{theorem}

\begin{proof}
By Proposition \ref{the theorem before 23.3.16}, it is enough to prove that each clause implies that $F$ is a uniform set-system. By Proposition \ref{1}, Clause $(1)$ implies that $F$ is a uniform set-system. By Proposition \ref{elimination of relevant} Clause $(3)$ implies that $F$ is a uniform set-system. But Clause $(2)$ implies Clause $(3)$.
\end{proof}

\begin{corollary}\label{corollary 2.6}
Let $G$ be a graph. The following are equivalent:
\begin{enumerate}
\item $G$ is a KE graph. \item For some non-empty HKE set-system $F \subseteq \Omega(G)$, there is a matching $M:V[G]-\bigcup F \to \bigcap F$ and Equality \ref{equality 1} holds for every two non-empty disjoint sub-set-systems, $\Gamma_1,\Gamma_2$ of $F$. \item For some non-empty HKE set-system $F \subseteq \Omega(G)$, there is a matching $M:V[G]-\bigcup F \to \bigcap F$ and Equality \ref{equality 1} holds for every two non-empty disjoint sub-set-systems, $\Gamma_1,\Gamma_2$ of $F$ with $\Gamma_1 \cup \Gamma_2=F$.
\end{enumerate}

\end{corollary}

\begin{proof}
By Theorem \ref{equivalent formulations of a KE set-system} and Theorem \ref{the main theorem of dam}.
\end{proof}

\bibliographystyle{amsplain}
\bibliography{..//..//..//lit}

\providecommand{\bysame}{\leavevmode\hbox to3em{\hrulefill}\thinspace}
\providecommand{\MR}{\relax\ifhmode\unskip\space\fi MR }
\providecommand{\MRhref}[2]{%
  \href{http://www.ams.org/mathscinet-getitem?mr=#1}{#2}
}
\providecommand{\href}[2]{#2}
\begin{thebibliography}{1}

\bibitem{broken}
Adi Jarden, \emph{The first time ke is broken up}, arXiv preprint
  arXiv:1603.06887 (2016).

\bibitem{hke}
\bysame, \emph{Hereditary konig egervary collections}, arXiv preprint
  arXiv:1603.06552 (2016).

\bibitem{dam}
Adi Jarden, Vadim~E Levit, and Eugen Mandrescu, \emph{Two more characterization
  of konig egervary graphs}, Submitted.

\bibitem{jlm}
\bysame, \emph{Monotonic properties of collections of maximum independent sets
  of a graph}, arXiv preprint arXiv:1506.00249 (2015).

\end{thebibliography}






\end{document}